\documentclass[journal, twoside, web]{ieeecolor}

\usepackage{lcsys}

\usepackage{graphicx}
\graphicspath{{Figures/}}
\usepackage{subcaption}

\let\labelindent\relax

\usepackage{array, multirow, booktabs}
\usepackage{balance}
\usepackage{cite}
\usepackage{enumitem}

\usepackage{mathtools, amssymb, bm, amsthm}

\newtheorem{proposition}{Proposition}

\theoremstyle{definition}
\newtheorem{problem}{Problem}

\newtheorem{remark}{Remark}

\usepackage{hyperref}
\hypersetup{colorlinks = true, linkcolor = blue, citecolor = blue, urlcolor = blue}

\newcommand{\mbf}[1]{\mathbf{#1}}

\DeclareMathOperator*{\T}{\mathsf{T}}
\DeclareMathOperator*{\Z}{\mathcal{Z}}

\DeclareMathOperator*{\Q}{\mathcal{Q}}
\DeclareMathOperator*{\R}{\mathcal{R}}
\DeclareMathOperator*{\mL}{\mathcal{L}}

\DeclareMathOperator*{\blkdiag}{\mathrm{blkdiag}}

\allowdisplaybreaks

\begin{document}
\title{Mode-Prefix-Based Control of Switched Linear Systems with Applications to Fault Tolerance}
\author{Ram Padmanabhan, \IEEEmembership{Graduate Student Member, IEEE}, Antoine Aspeel,\\Necmiye Ozay, \IEEEmembership{Senior Member, IEEE}, and Melkior Ornik, \IEEEmembership{Senior Member, IEEE}
\thanks{This work was supported by Air Force Office of Scientific Research grant FA9550-23-1-0131, NASA University Leadership Initiative grant 80NSSC22M0070, and ONR CLEVR-AI MURI (\#N00014-21-1-2431).}
\thanks{RP and MO are with the University of Illinois Urbana-Champaign, Urbana, IL 61801, USA. AA and NO are with the University of Michigan, Ann Arbor, MI 48109, USA. (Corresponding Author: Ram Padmanabhan. Email: ramp3@illinois.edu.)}
}

\maketitle

\begin{abstract}
In this paper, we consider the problem of designing prefix-based optimal controllers for switched linear systems over finite horizons. This problem arises in fault-tolerant control, when system faults result in abrupt changes in dynamics. We consider a class of mode-prefix-based linear controllers that depend only on the history of the switching signal. The proposed optimal control problems seek to minimize both expected performance and worst-case performance over switching signals. We show that this problem can be reduced to a convex optimization problem. To this end, we synthesize one controller for each switching signal under a prefix constraint that ensures consistency between controllers. Then, system level synthesis is used to obtain a convex program in terms of the system-level parameters. In particular, it is shown that the prefix constraints are linear in terms of the system-level parameters.
Finally, we apply this framework for optimal control of a fighter jet model suffering from system faults, illustrating how fault tolerance is ensured.
\end{abstract}


\section{Introduction} \label{sec:Introduction}
\IEEEPARstart{C}{ontrol} systems are required to be fault-tolerant and adaptive to changes in dynamics, often caused by the failure of components such as sensors or actuators. Failures that cause abrupt changes in dynamics can lead to undesirable consequences, preventing a control system from achieving objectives such as reachability or optimality \cite{AH19}. Classical control strategies such as robust control and adaptive control may not succeed in mitigating these faults \cite{WZ01, AD08}. Such issues have also contributed to the development of techniques to detect abrupt changes in dynamical systems \cite{ASW85,harirchi2018guaranteed}. Numerous techniques for fault-tolerant control of systems are described in \cite{AH19}, and a number of optimal control approaches to ensure fault tolerance have also been developed, e.g., \cite{FTC1, FTC2, FTC3}.

An abrupt change in dynamics can be modeled through the framework of \emph{switched} systems \cite{DL03}, which are governed by a switching signal determining the dynamics to be followed at each time instant. When a fault occurs, a system switches from nominal behavior to malfunctioning behavior. In such a framework, fault tolerance can be ensured by designing controllers for switched systems by achieving a specified objective such as optimality \cite{Sw1, Sw2, Sw3}.

\begin{figure}
    \centering
    \includegraphics[width=0.49\textwidth]{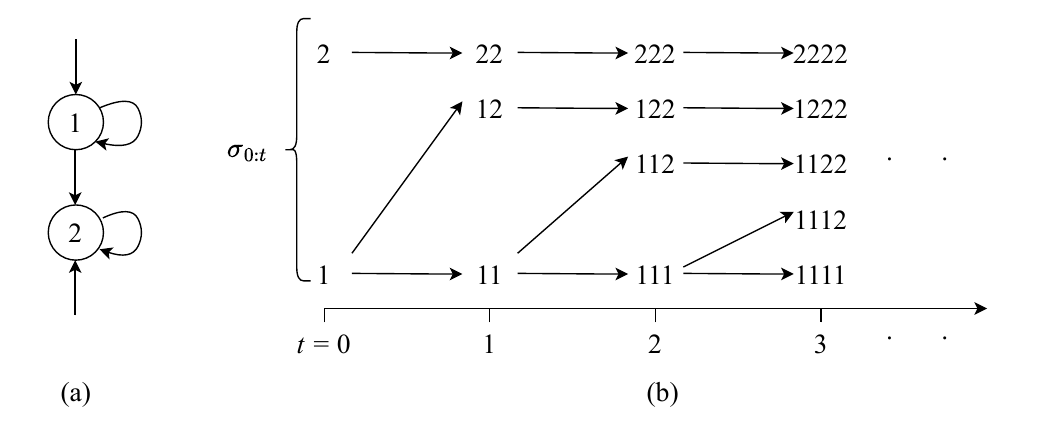}
    \caption{{(a) A fault model representing a transition from nominal to faulty operation at some \emph{a priori} unknown time instant. (b) Rolled out version of the fault model in (a), showing possible switching signal prefixes over time.}}
    \label{fig:FSM}
    \vspace{-0.7cm}
\end{figure}

Motivated by the applications to fault tolerance, in this paper we consider the problem of designing an optimal controller for discrete-time switched linear systems over a finite horizon. We consider linear controllers that are \emph{mode-prefix-based}, i.e., they depend only on the history or prefix of the switching signal available at each time instant. {Throughout this paper, we use the term `prefix' to indicate the history of the switching signal available at a given time instant. Fig.~\ref{fig:FSM} illustrates an example of how this prefix evolves over time in the context of a system switching from nominal operation (denoted $1$) to faulty operation (denoted $2$).} We formulate optimal control problems depending on whether the switching signal is chosen stochastically or {non-}deterministically from a prescribed \emph{language}, and depending on the characteristics of process and measurement noises. As such problems are commonly non-convex in the controller gains, we invoke the framework of system level synthesis \cite{SLS19}, using a nonlinear change of variables to reformulate the optimal control problems as convex programs. 
{In the context of fault-tolerant control since it is not known \emph{a priori} which fault will occur when, the controller gains corresponding to two different switching signals need to be identical if the histories of the switching signals they observe are identical to ensure causality. This can be encoded with a linear prefix constraint.} 
A key contribution of our work is in proving that the prefix constraint remains linear even after the transformation through system level synthesis. 

A closely related notion is the property of \emph{quadratic invariance} of constraints, which guarantees convexity of a constraint in the Youla parameterization \cite{SB10} if the constraint is originally convex \cite{LL11}. Such a property has been proved for prefix constraints in \cite{RYO20} for state estimation with missing measurements. However, \cite{RYO20} considers systems that can switch between exactly two modes, while we allow for arbitrary switching between multiple modes. Further, we parameterize the controller using system level synthesis, and not the Youla parameterization used in \cite{RYO20}. Using system level synthesis, the controller can be implemented without inverting the nonlinear change of variables \cite{SLS19}, unlike the Youla parameterization, thus avoiding numerical and computational issues. Prefix-based constraints have been used in \cite{YO18} for fault tolerance with temporal logic specifications in the context of open-loop planning instead of closed-loop control. In comparison to existing techniques to ensure fault tolerance \cite{AH19, FTC1, FTC2, FTC3} and the works mentioned above, our framework can address more complex faulty behavior in control systems modeled by multiple failure modes without \emph{a priori} knowledge of failure times.


The remainder of this paper is organized as follows. In Section \ref{sec:Preliminaries}, we introduce basic notation and provide a brief overview of finite-horizon system level synthesis. In Section~\ref{sec:Formulation}, we present the finite-horizon optimal control problems we aim to solve for switched linear systems, and impose linear prefix constraints on the controller. Using system level synthesis, we reformulate these problems as convex programs in Section \ref{sec:Reformulation}, proving prefix constraints remain linear after the change of variables in system level synthesis. In Section \ref{sec:Examples}, we present applications of these convex programs to fault tolerance in a model of a fighter jet.

\section{Preliminaries} \label{sec:Preliminaries}

\subsection{Notation}
Let $X$ and $Y$ be sets. Then, $Y^X$ denotes the set of all functions from $X$ to $Y$. The quantity $I_n$ denotes the $n\times n$ identity matrix, and the subscript $n$ is dropped when dimensions are clear from context. The quantity $\bm{0}$ denotes the zero matrix of appropriate dimensions. Let $\mbf{M} = \begin{bmatrix} M_{00} & \bm{0} & \ldots & \bm{0} \\ M_{10} & M_{11} & \ldots & \bm{0} \\ \vdots & \vdots & \ddots & \vdots \\ M_{m0} & M_{m1} & \ldots & M_{mn} \end{bmatrix}$ be a block lower triangular matrix. Then, the matrix $\mbf{M}(:\!t) = \begin{bmatrix} M_{00} & \ldots & \bm{0} \\ \vdots & \ddots & \vdots \\ M_{t0} & \ldots & M_{tt}\end{bmatrix}$ denotes the first $(t+1)$ row and column \emph{blocks} of $\mbf{M}$.

\subsection{An Overview of System Level Synthesis}
We now provide a brief overview of finite-horizon system level synthesis \cite{SLS19}, an approach to design optimal controllers using system responses. Consider a linear, time-varying system
\begin{equation} \label{eq:LTV}
x_{t+1} = A_tx_t + B_tu_t + w_t, ~~~ y_t = C_tx_t + v_t
\end{equation}
over the horizon $t = 0, \ldots, T$. Here, $x_t \in \mathbb{R}^n$ is the state, $u_t \in \mathbb{R}^p$ is the control input, $w_t \in \mathbb{R}^n$ is the process noise, $y_t \in \mathbb{R}^{m}$ is the measurement and $v_t \in \mathbb{R}^m$ is the measurement noise. Consider linear time-varying, output-feedback controllers
\begin{equation} \label{eq:u_SLS}
u_t = \sum_{\tau \leq t} K_{(t, \tau)}y_{\tau},
\end{equation}
where the controller gains $K_{(t, \tau)}$ are designed according to a specified optimal control problem. Introduce the quantities $\mbf{x} \coloneqq \left[x_{0}^{\T}, \ldots, x_{T}^{\T} \right]^{\T}, ~\mbf{y} \coloneqq \left[y_{0}^{\T}, \ldots, y_{T}^{\T} \right]^{\T}, \mbf{w} \coloneqq \left[x_{0}^{\T}, w_{0}^{\T}, \ldots, w_{T-1}^{\T}\right]^{\T}, ~\mbf{v} \coloneqq \left[v_{0}^{\T}, \ldots, v_{T}^{\T}\right]^{\T}, \mbf{u} \coloneqq \left[u_{0}^{\T}, \ldots, u_{T}^{\T}\right]^{\T}, \mbf{A} \coloneqq \blkdiag\{A_0, \ldots, A_{T-1}, \bm{0}\}, \mbf{B} \coloneqq \blkdiag\{B_0, \ldots, B_{T-1}, \bm{0}\}, \mbf{C} \coloneqq \blkdiag\{C_0, \ldots, C_T\},$
\begin{equation} \label{eq:vec}
\mbf{K} \coloneqq 
\begin{bmatrix} 
K_{(0,0)} & 0 & 0 & \ldots & 0 \\
K_{(1,0)} & K_{(1,1)} & 0 & \ldots & 0 \\
\vdots & \vdots & \vdots & \ddots & \vdots \\ 
K_{(T,0)} & K_{(T,1)} & K_{(T,2)} & \ldots & K_{(T,T)} 
\end{bmatrix},
\end{equation}
so that the controller \eqref{eq:u_SLS} is written as $\mbf{u} = \mbf{Ky}$. Let $\Z$ be the block downshift operator, a matrix with {$T$ copies of $I_n$} on its first block subdiagonal and zeros elsewhere{, so that $\Z \in \mathbb{R}^{n(T+1)\times n(T+1)}$.} Using the above definitions, the behavior of the system \eqref{eq:LTV} can be written as:
\begin{align} \label{eq:LTV_vec}
\begin{split}
\mbf{x} &= \Z\left(\mbf{Ax} + \mbf{Bu}\right) + \mbf{w}, \\
\mbf{y} &= \mbf{Cx} + \mbf{v}.
\end{split}
\end{align}
Substituting $\mbf{u} = \mbf{Ky}$ and solving for $\mbf{x}, \mbf{y}$ and subsequently $\mbf{u}$, we obtain the following \emph{system response} map \cite{SLS19}:
\begin{equation} \label{eq:SR_map}
\begin{bmatrix} \mbf{x} \\ \mbf{u} \end{bmatrix} = \mbf{\Phi} \begin{bmatrix} \mbf{w} \\ \mbf{v} \end{bmatrix} \coloneqq \begin{bmatrix} \mbf{\Phi}_{xx} & \mbf{\Phi}_{xy} \\ \mbf{\Phi}_{ux} & \mbf{\Phi}_{uy} \end{bmatrix} \begin{bmatrix} \mbf{w} \\ \mbf{v} \end{bmatrix},
\end{equation}
where $\mbf{\Phi}_{xx} = (I - \Z(\mbf{A}+\mbf{BKC}))^{-1}$, $\mbf{\Phi}_{xy} = \mbf{\Phi}_{xx} \Z\mbf{BK}$, $\mbf{\Phi}_{ux} = \mbf{KC\Phi}_{xx}$ and $\mbf{\Phi}_{uy} = \mbf{K} + \mbf{KC\Phi}_{xx}\Z\mbf{BK}$. Since $\mbf{A}, \mbf{B}, \mbf{C}, \mbf{K}$ and $\Z$ are block diagonal or block lower triangular matrices, it is clear that the maps $\{\mbf{\Phi}_{xx}, \mbf{\Phi}_{xy}, \mbf{\Phi}_{ux}, \mbf{\Phi}_{uy}\}$ are also block lower triangular matrices.

Rather than optimizing over the controller parameterized by $\mbf{K}$ for optimal control, the motivation that underlies system level synthesis is to optimize directly over these system response maps. In this context, it was shown in \cite{HY22} that the affine subspace defined by
\begin{align} \label{eq:Affine}
	\begin{bmatrix} I - \Z\mbf{A} ~ -\Z\mbf{B} \end{bmatrix}
	\mbf{\Phi} = 
    \begin{bmatrix} I ~~ \bm{0} \end{bmatrix}, \mbf{\Phi} \begin{bmatrix} I-\Z\mbf{A} \\ -\mbf{C} \end{bmatrix} = 
    \begin{bmatrix} I \\ \bm{0} \end{bmatrix}
\end{align}
parameterizes all possible system responses \eqref{eq:SR_map}. It was also shown that for any block lower triangular matrices $\{\mbf{\Phi}_{xx}, \mbf{\Phi}_{xy}, \mbf{\Phi}_{ux}, \mbf{\Phi}_{uy}\}$ satisfying \eqref{eq:Affine}, the controller $\mbf{K} = \mbf{\Phi}_{uy} - \mbf{\Phi}_{ux}\mbf{\Phi}_{xx}^{-1}\mbf{\Phi}_{xy}$ achieves the desired response. 

\section{Problem Formulation} \label{sec:Formulation}
We consider a set $\left\{\Sigma_i\right\}_{i=1}^{M}$ of $M$ systems with the following linear, time-varying dynamics on the horizon $t = 0, \ldots, T$:
\begin{equation} \label{eq:System_i}
	\Sigma_i : x_{t+1} = A^{i}_{t}x_t + B^{i}_{t}u_t + w^{i}_{t}, ~~~ y_t = C^{i}_{t}x_t + v^{i}_{t}
\end{equation}
for $i = 1, \ldots, M$, where the state $x_t \in \mathbb{R}^n$, input $u_t \in \mathbb{R}^p$ and output $y_t \in \mathbb{R}^m$. For each system $\Sigma_i$, $w^{i}_{t} \in \mathbb{R}^n$ is the process noise and $v^{i}_{t} \in \mathbb{R}^m$ is the measurement noise. Both $w^{i}_{t}$ and $v^{i}_{t}$ are either stochastic, or {non-}deterministic but bounded. 

The mechanism of switching between these systems is governed by a \emph{switching signal} $\sigma: \{0, \ldots, T\} \to \{1, \ldots, M\}$, which determines the dynamics that are followed at each time instant $t$. The set of possible switching signals is the \emph{language} $\mL \subseteq \{1, \ldots, M\}^{\{0, \ldots, T\}}$, and $\sigma$ is chosen either stochastically or {non-}deterministically from $\mL$. For a given $\sigma \in \mL$, the dynamics \eqref{eq:System_i} for $i = 1, \ldots, M$ can be compactly written as
\begin{equation} \label{eq:System_sigma}
x_{t+1} = A_{t}^{\sigma_t}x_t + B_{t}^{\sigma_t}u_t + w_{t}^{\sigma_t}, ~~~ y_t = C_{t}^{\sigma_t}x_t + v_{t}^{\sigma_t}
\end{equation}
where $\sigma_t = i \in \{1, \ldots, M\}$ determines the system $i$ whose dynamics are followed at each time instant $t$. {For example, the language $\mL$ corresponding to the fault model in Fig.~\ref{fig:FSM} is:
\begin{align}
    \mL = \{&\sigma\in\{1,2\}^{\{0,\dots,T\}} : \exists~ t_{\mathrm{fault}} \in \{0, \ldots, T\} \text{ s.t. } \nonumber \\
    &\sigma_t=1 \text{ if } t < t_{\mathrm{fault}},
    \text{ and } \sigma_t=2 \text{ otherwise}\}. \label{eq:eg_L}
\end{align}
Moreover, when the transitions in the fault model are equipped with probabilities, i.e., the fault model is a Markov chain, each switching signal $\sigma$ in the language $\mL$ has an associated probability $\pi(\sigma)$.}


We consider \emph{mode-prefix-based} linear output-feedback controllers with memory, of the form
\begin{equation} \label{eq:u_K_OF}
u_t = \sum_{\tau \leq t} K_{(t, \tau)}^{\sigma_{0:t}} y_{\tau},
\end{equation}
where we note that only the prefix of the switching signal $\sigma_{0:t}$ is available to the controller at time $t$. Recall the notations $\mbf{x}, \mbf{u}, \mbf{y}$ from Section \ref{sec:Preliminaries}, and further let 
\begin{align*}
\mbf{w}^{\sigma} &\coloneqq \left[\left({x_{0}^{\sigma_0}}\right)^{\T}, \left({w_{0}^{\sigma_0}}\right)^{\T}, \left({w_{1}^{\sigma_1}}\right)^{\T}, \ldots, \left({w_{T-1}^{\sigma_{T-1}}}\right)^{\T}\right]^{\T}, \\
\text{and }~ \mbf{v}^{\sigma} &\coloneqq \left[\left({v_{0}^{\sigma_0}}\right)^{\T}, \ldots, \left({v_{T}^{\sigma_T}}\right)^{\T}\right]^{\T},
\end{align*}
denoting process noise and measurement noise vectors over time induced by a particular switching signal $\sigma$. Similarly, define $\mbf{A^{\sigma}} \coloneqq \blkdiag\{A_{0}^{\sigma_0}, \ldots, A_{T-1}^{\sigma_{T-1}}, \bm{0}\}$, $\mbf{B^{\sigma}} \coloneqq \blkdiag\{B_{0}^{\sigma_0}, \ldots, B_{T-1}^{\sigma_{T-1}}, \bm{0}\}$, and $\mbf{C^{\sigma}} \coloneqq \blkdiag\{C_{0}^{\sigma_0}, \ldots, C_{T}^{\sigma_{T}}\}$
as quantities denoting the system matrices over time, induced by a particular switching signal $\sigma$. Then, the behavior of the dynamics \eqref{eq:System_sigma} can be rewritten as
\begin{align} \label{eq:System_vec}
\begin{split}
\mbf{x} &= \Z\left(\mbf{A^{\sigma}}\mbf{x} + \mbf{B^{\sigma}}\mbf{u}\right) + \mbf{w}^{\sigma}, \\
\mbf{y} &= \mbf{C^{\sigma}}\mbf{x} + \mbf{v}^{\sigma}.
\end{split}
\end{align}
In this setting, we are interested in two problems.

\begin{problem}[$\mathcal{H}_2$-optimal control, stochastic switching] \label{prob:H2}
Assume that the process noise and measurement noise are stochastic and normally distributed, i.e., $\mbf{w}^{\sigma} \sim \mathcal{N}(\bm{0}, \mbf{P_{w}^{\sigma}})$ and $\mbf{v}^{\sigma} \sim \mathcal{N}(\bm{0}, \mbf{P_{v}^{\sigma}})$ for each $\sigma \in \mL$, where $\mbf{P_{w}^{\sigma}} = \blkdiag\{\mbf{P}_{x_0}^{\sigma_0}, \mbf{P}_{w_0}^{\sigma_0}, \ldots, \mbf{P}_{w_{T-1}}^{\sigma_{T-1}}\}$ and $\mbf{P_{v}^{\sigma}} = \blkdiag\{\mbf{P}_{v_0}^{\sigma_0}, \ldots, \mbf{P}_{v_T}^{\sigma_T}\}$ based on the definitions for $\mbf{w}^{\sigma}$ and $\mbf{v}^{\sigma}$. Further assume that the switching signal $\sigma$ is chosen stochastically from $\mL$ with probability $\pi(\sigma)$, such that $\sum_{\sigma \in \mL} \pi(\sigma) = 1$. Then, we are interested in the following optimal control problem:
\begin{align} \label{eq:H2}
\begin{split}
&\min_{\left(K^{\sigma_{0:t}}_{(t,\tau)}\right)_{t=0,\dots,T}^{\tau=0,\dots,t}} ~\mathbb{E}_{\sigma \in \mL} ~\mathbb{E} \left[\sum_{t = 0}^{T} x_{t}^{\T}Q_tx_t + u_{t}^{\T}R_tu_t\right] \\
&\text{subject to the dynamics \eqref{eq:System_sigma} and the controller \eqref{eq:u_K_OF},}
\end{split}
\end{align}
where the inner expectation is over $\mbf{w}^{\sigma}$ and $\mbf{v}^{\sigma}$.
\end{problem}


\begin{problem}[$\mathcal{L}_{1}$-optimal control, worst-case switching] \label{prob:L1}
Assume that the process noise and measurement noise are {non-}deterministic and bounded, i.e., $\left\|\mbf{w}^{\sigma}\right\|_{\infty} \leq \overline{w}$ and $\left\|\mbf{v}^{\sigma}\right\|_{\infty} \leq \overline{v}$ for each $\sigma \in \mL$, also chosen {non-}deterministically. We are interested in the following problem which seeks to minimize the maximum deviation of the state trajectory from the origin:
\begin{align} \label{eq:L1}
\begin{split}
&\min_{\left(K^{\sigma_{0:t}}_{(t,\tau)}\right)_{t=0,\dots,T}^{\tau=0,\dots,t}} \max_{\sigma \in \mL} \max_{\substack{\left\|\mbf{w}^{\sigma}\right\|_{\infty} \leq \overline{w} \\ \left\|\mbf{v}^{\sigma}\right\|_{\infty} \leq \overline{v}}} \|\mbf{x}\|_{\infty} \\
&\text{subject to the dynamics \eqref{eq:System_sigma} and the controller \eqref{eq:u_K_OF}.}
\end{split}
\end{align}
\end{problem}

We note that Problems \ref{prob:H2} and \ref{prob:L1} are two specific optimal control problems that we consider. The framework we develop can generalize to other optimal control problems with convex costs and constraints, as discussed in Section \ref{sec:Reformulation}.

To solve Problems \ref{prob:H2} and \ref{prob:L1}, we consider a controller for each switching sequence $\sigma$, i.e., $K_{(t,\tau)}^\sigma$, and enforce that same prefixes must result in the same controller, i.e.,
\begin{equation} \label{eq:prefix_Ka}
\sigma_{0:t} = \sigma'_{0:t} \implies K^{\sigma}_{(t, \tau)} = K^{\sigma'}_{(t, \tau)} \text{ for any $\sigma,\sigma'\in\mathcal{L}$, $\tau \leq t$.}
\end{equation}
This allows to synthesize one controller for each switching sequence while ensuring consistency.


We remark that such a constraint is similar to those used in \cite{RYO20}. Given a particular switching signal $\sigma \in \mL$, define
\begin{equation} \label{eq:K_sigma}
\mbf{K^{\sigma}} \coloneqq \begin{bmatrix} 
K^{\sigma}_{(0,0)} & 0 & 0 & \ldots & 0 \\
K^{\sigma}_{(1,0)} & K^{\sigma}_{(1,1)} & 0 & \ldots & 0 \\
\vdots & \vdots & \vdots & \ddots & \vdots \\ 
K^{\sigma}_{(T,0)} & K^{\sigma}_{(T,1)} & K^{\sigma}_{(T,2)} & \ldots & K^{\sigma}_{(T,T)} 
\end{bmatrix},
\end{equation}
similar to \eqref{eq:vec}. Then, the prefix constraint \eqref{eq:prefix_Ka} is written as
\begin{equation} \label{eq:prefix_K}
\sigma_{0:t} = \sigma'_{0:t} \implies \mbf{K}^{\sigma}(:\!t) = \mbf{K}^{\sigma'}(:\!t) \text{ for any $t$},
\end{equation}
and the controller \eqref{eq:u_K_OF} can be written as $\mbf{u} = \mbf{K^{\sigma}}\mbf{y}$. While constraint \eqref{eq:prefix_K} is linear in $\mathbf{K}^\sigma$, the objectives in \eqref{eq:H2} and \eqref{eq:L1} are not convex in $\mbf{K^{\sigma}}$. In the next section, we use system level synthesis to reformulate these as convex problems using prefix constraints similar to \eqref{eq:prefix_K} on the system response maps.

\section{Convex Reformulation} \label{sec:Reformulation}
In this section, we demonstrate how Problems \ref{prob:H2} and \ref{prob:L1} can be rewritten as convex problems using the system response maps. Substituting the controller $\mbf{u} = \mbf{K^{\sigma}}\mbf{y}$ in \eqref{eq:System_vec}, we obtain the mode-prefix-based system response map:
\begin{equation} \label{eq:SR_map_sigma}
\begin{bmatrix} \mbf{x} \\ \mbf{u} \end{bmatrix} = \mbf{\Phi^{\sigma}} \begin{bmatrix} \mbf{w}^{\sigma} \\ \mbf{v}^{\sigma} \end{bmatrix} \coloneqq \begin{bmatrix} \mbf{\Phi}^{\sigma}_{xx} & \mbf{\Phi}^{\sigma}_{xy} \\ \mbf{\Phi}^{\sigma}_{ux} & \mbf{\Phi}^{\sigma}_{uy} \end{bmatrix} \begin{bmatrix} \mbf{w}^{\sigma} \\ \mbf{v}^{\sigma} \end{bmatrix},
\end{equation}
where $\{\mbf{\Phi}^{\sigma}_{xx}, \mbf{\Phi}^{\sigma}_{xy}, \mbf{\Phi}^{\sigma}_{ux}, \mbf{\Phi}^{\sigma}_{uy}\}$ are block lower triangular maps defined using $\mbf{A^{\sigma}}$, $\mbf{B^{\sigma}}$, $\mbf{C^{\sigma}}$ and $\mbf{K^{\sigma}}$, similar to \eqref{eq:SR_map}. We now show that the prefix constraint \eqref{eq:prefix_K} translates to similar prefix constraints on $\{\mbf{\Phi}^{\sigma}_{xx}, \mbf{\Phi}^{\sigma}_{xy}, \mbf{\Phi}^{\sigma}_{ux}, \mbf{\Phi}^{\sigma}_{uy}\}$.

\begin{proposition} \label{prop:QI}
For two switching signals $\sigma, \sigma' \in \mL$ such that $\sigma_{0:t} = \sigma'_{0:t}$, the constraint $\mbf{K}^{\sigma}(:\!t) = \mbf{K}^{\sigma'}(:\!t)$ for all $t$ holds if and only if $\phi^{\sigma}(:\!t) = \phi^{\sigma'}(:\!t)$ for all $t$, and for all $\phi^{\sigma}\in\{\mbf{\Phi}^{\sigma}_{xx}, \mbf{\Phi}^{\sigma}_{xy}, \mbf{\Phi}^{\sigma}_{ux}, \mbf{\Phi}^{\sigma}_{uy}\}$.
\end{proposition}

\begin{proof}
We first show the forward direction of the proposition. Suppose that for $\sigma, \sigma' \in \mL$ such that $\sigma_{0:t} = \sigma'_{0:t}$, the prefix constraint $\mbf{K}^{\sigma}(:\!t) = \mbf{K}^{\sigma'}(:\!t)$ holds for all $t$. Note that when $\sigma_{0:t} = \sigma'_{0:t}$, the dynamics are identical up to time $t$, i.e.,
\begin{equation} \label{eq:prefix_ABC}
\mbf{A^{\sigma}}(:\!t) = \mbf{A}^{\sigma'}(:\!t),  \mbf{B^{\sigma}}(:\!t) = \mbf{B}^{\sigma'}(:\!t),  \mbf{C^{\sigma}}(:\!t) = \mbf{C}^{\sigma'}(:\!t).
\end{equation}
Let $\tilde{\mbf{A}}^{\sigma} = \mbf{A^{\sigma}} + \mbf{B^{\sigma}}\mbf{K^{\sigma}}\mbf{C^{\sigma}}$. Then, constraints \eqref{eq:prefix_K} and \eqref{eq:prefix_ABC} imply $\tilde{\mbf{A}}^{\sigma}(:\!t) = \tilde{\mbf{A}}^{\sigma'}(:\!t)$. Next, $\mbf{\Phi}_{xx}^{\sigma} = (I - \Z\tilde{\mbf{A}}^{\sigma})^{-1}$. Since block-triangularity is preserved through matrix inversion,
\begin{align*}
\mbf{\Phi}_{xx}^{\sigma}(:\!t) &= (I(:\!t) - \Z(:\!t)\tilde{\mbf{A}}^{\sigma}(:\!t))^{-1} \\
&= (I(:\!t) - \Z(:\!t)\tilde{\mbf{A}}^{\sigma'}(:\!t))^{-1} = \mbf{\Phi}_{xx}^{\sigma'}(:\!t),
\end{align*}
representing the prefix constraint on $\mbf{\Phi}_{xx}^{\sigma}$. Similarly, using $\mbf{\Phi}_{xy}^{\sigma} = \mbf{\Phi}_{xx}^{\sigma}\Z\mbf{B^{\sigma}K^{\sigma}}$,
\begin{align*}
\mbf{\Phi}_{xy}^{\sigma}(:\!t) &= \mbf{\Phi}_{xx}^{\sigma}(:\!t)\Z(:\!t)\mbf{B^{\sigma}}(:\!t)\mbf{K^{\sigma}}(:\!t) \\
&= \mbf{\Phi}_{xx}^{\sigma'}(:\!t)\Z(:\!t)\mbf{B^{\sigma'}}(:\!t)\mbf{K^{\sigma'}}(:\!t) = \mbf{\Phi}_{xy}^{\sigma'}(:\!t),
\end{align*}
representing the prefix constraint on $\mbf{\Phi}_{xy}^{\sigma}$. Using identical arguments, prefix constraints must also hold for $\mbf{\Phi}_{ux}^{\sigma}$ and $\mbf{\Phi}_{uy}^{\sigma}$ when $\mbf{K^{\sigma}}(:\!t) = \mbf{K}^{\sigma'}(:\!t)$, thus proving the forward direction.

Conversely, suppose that the prefix constraint $\phi^{\sigma}(:\!t) = \phi^{\sigma'}(:\!t)$ holds for each $\phi^{\sigma} \in \{\mbf{\Phi}^{\sigma}_{xx}, \mbf{\Phi}^{\sigma}_{xy}, \mbf{\Phi}^{\sigma}_{ux}, \mbf{\Phi}^{\sigma}_{uy}\}$ when $\sigma_{0:t} = \sigma'_{0:t}$. Using \eqref{eq:Affine}, $\mbf{K}^{\sigma} = \mbf{\Phi}^{\sigma}_{uy} - \mbf{\Phi}_{ux}^{\sigma}{\mbf{\Phi}_{xx}^{\sigma}}^{-1}\mbf{\Phi}_{xy}^{\sigma}$. Since block-triangularity is preserved through inversion,
\begin{align*}
\mbf{K}^{\sigma}(:\!t) &= \mbf{\Phi}^{\sigma}_{uy}(:\!t) - \mbf{\Phi}_{ux}^{\sigma}(:\!t)\mbf{\Phi}_{xx}^{\sigma}(:\!t)^{-1}\mbf{\Phi}_{xy}^{\sigma}(:\!t) \\
&= \mbf{\Phi}^{\sigma'}_{uy}(:\!t) - \mbf{\Phi}_{ux}^{\sigma'}(:\!t)\mbf{\Phi}_{xx}^{\sigma'}(:\!t)^{-1}\mbf{\Phi}_{xy}^{\sigma'}(:\!t) = \mbf{K}^{\sigma'}(:\!t),
\end{align*}
proving the reverse direction of the proposition. Thus, the prefix constraint \eqref{eq:prefix_K} is equivalent to:
\begin{align} \label{eq:prefix_Phi}
\begin{split}
\sigma_{0:t} = \sigma'_{0:t} \implies &\phi^{\sigma}(:\!t) = \phi^{\sigma'}(:\!t) \\
\text{ for each } &\phi^{\sigma} \in \{\mbf{\Phi}^{\sigma}_{xx}, \mbf{\Phi}^{\sigma}_{xy}, \mbf{\Phi}^{\sigma}_{ux}, \mbf{\Phi}^{\sigma}_{uy}\},
\end{split}
\end{align}
concluding the proof.
\end{proof}

\begin{remark}
For a constraint that is convex in the original controller parameterization, the property of quadratic invariance guarantees that convexity is preserved in the Youla parameterization as well \cite{LL11}. In Proposition \ref{prop:QI}, we show that the linear constraint \eqref{eq:prefix_K} is not just convex, but in fact \emph{linear} in the system level parameters, by exploiting the structure of the matrices $\{\mbf{\Phi}^{\sigma}_{xx}, \mbf{\Phi}^{\sigma}_{xy}, \mbf{\Phi}^{\sigma}_{ux}, \mbf{\Phi}^{\sigma}_{uy}\}$. {This proof is the primary contribution of this section.}
\end{remark}

\begin{remark}[Time delays and Bounded Memory]
It is worth mentioning that time delays in observing the prefix of $\sigma$ can also be accommodated by adjusting the prefix constraint \eqref{eq:prefix_Phi} appropriately. If we observe only $\sigma_{0:t-d}$ at time $t$ for some delay $d$, we can impose the prefix constraint $\phi^{\sigma}(:\!t) = \phi^{\sigma'}(:\!t)$ when $\sigma_{0:t-d} = \sigma'_{0:t-d}$ for each $\phi^{\sigma} \in \{\mbf{\Phi}^{\sigma}_{xx}, \mbf{\Phi}^{\sigma}_{xy}, \mbf{\Phi}^{\sigma}_{ux}, \mbf{\Phi}^{\sigma}_{uy}\}$. This constraint guarantees that the controllers are identical up to time $t$. {Similarly, constraints on the available memory of measurements and prefix of switching signal can be handled using sparsity or prefix-like constraints, respectively.}
\end{remark}

We note that the prefix constraints \eqref{eq:prefix_Phi} are linear in the system response maps $\{\mbf{\Phi}^{\sigma}_{xx}, \mbf{\Phi}^{\sigma}_{xy}, \mbf{\Phi}^{\sigma}_{ux}, \mbf{\Phi}^{\sigma}_{uy}\}$. Based on \eqref{eq:Affine}, the map $\mbf{\Phi^{\sigma}}$ in \eqref{eq:SR_map_sigma} must also satisfy the following affine constraints for each $\sigma \in \mL$:
\begin{align} \label{eq:Affine_sigma}
	\begin{bmatrix} I - \Z\mbf{A^{\sigma}} ~ -\Z\mbf{B^{\sigma}} \end{bmatrix} \mbf{\Phi^{\sigma}} = \begin{bmatrix} I ~~ \bm{0} \end{bmatrix}, \mbf{\Phi^{\sigma}} \begin{bmatrix} I-\Z\mbf{A^{\sigma}} \\ -\mbf{C^{\sigma}} \end{bmatrix} = 
    \begin{bmatrix} I \\ \bm{0} \end{bmatrix},
\end{align}
thus parameterizing all possible system responses \eqref{eq:SR_map_sigma}. We now use this framework to rewrite Problems \ref{prob:H2} and \ref{prob:L1} as convex optimization problems.

\subsection{Problem 1: $\mathcal{H}_2$-optimal control, stochastic switching} \label{sec:H2}

To rewrite Problem~\ref{prob:H2} as a convex program, we adapt a procedure proposed in~\cite[Sec.~2.2.2]{SLS19} for the Linear Quadratic Regulator problem. Define the quantities $\Q \coloneqq \blkdiag\{Q_0, \ldots, Q_T\}$ and $\R \coloneqq \blkdiag\{R_0, \ldots, R_T\}$. Then,
\begin{align}
\mathbb{E} &\left[\sum_{t = 0}^{T} x_{t}^{\T}Q_tx_t + u_{t}^{\T}R_tu_t\right] = \mathbb{E} \left\|\begin{bmatrix} \Q^{\frac{1}{2}} & \bm{0} \\ \bm{0} & \R^{\frac{1}{2}} \end{bmatrix} \begin{bmatrix} \mbf{x} \\ \mbf{u} \end{bmatrix}\right\|_{2}^{2} \nonumber \\
= \mathbb{E} &\left\|\begin{bmatrix} \Q^{\frac{1}{2}} & \bm{0} \\ \bm{0} & \R^{\frac{1}{2}} \end{bmatrix} \mbf{\Phi^{\sigma}} \begin{bmatrix} \mbf{w^{\sigma}} \\ \mbf{v^{\sigma}} \end{bmatrix} \right\|_{2}^{2} \nonumber \\
= &\left\|\begin{bmatrix} \Q^{\frac{1}{2}} & \bm{0} \\ \bm{0} & \R^{\frac{1}{2}} \end{bmatrix} \mbf{\Phi^{\sigma}} \begin{bmatrix} \mbf{P_{w}^{\sigma}} & \bm{0} \\ \bm{0} & \mbf{P_{v}^{\sigma}} \end{bmatrix} \right\|_{F}^{2},
\end{align}
using \eqref{eq:SR_map_sigma} and the distributions of $\mbf{w^{\sigma}}$ and $\mbf{v^{\sigma}}$. Then, using the distribution $\pi(\sigma)$ on $\sigma \in \mL$, \eqref{eq:H2} reduces to:
\begin{align} \label{eq:H2_phi}
\begin{split}
&\min_{\{\mbf{\Phi^{\sigma}}\}_{\sigma \in \mL}} ~\sum_{\sigma \in \mL} \pi(\sigma) \left\|\begin{bmatrix} \Q^{\frac{1}{2}} & \bm{0} \\ \bm{0} & \R^{\frac{1}{2}} \end{bmatrix} \mbf{\Phi^{\sigma}} \begin{bmatrix} \mbf{P_{w}^{\sigma}} & \bm{0} \\ \bm{0} & \mbf{P_{v}^{\sigma}} \end{bmatrix} \right\|_{F}^{2} \\
&\text{subject to constraints \eqref{eq:prefix_Phi}, \eqref{eq:Affine_sigma} for each $\sigma \in \mL$}.
\end{split}
\end{align}
Equation \eqref{eq:H2_phi} is a convex quadratic program in the decision variables $\mbf{\Phi^{\sigma}}$ for $\sigma \in \mL$ with linear constraints, and can thus be solved efficiently.

\subsection{Problem 2: $\mathcal{L}_1$-optimal control, worst-case switching} \label{sec:L1}
To rewrite Problem~\ref{prob:L1} as a convex program, we adapt a procedure proposed in~\cite[Sec.~2.2.4]{SLS19} for a state-feedback controller. Let $\mbf{\Phi^{\sigma}_{x}} = \begin{bmatrix} \mbf{\Phi}^{\sigma}_{xx} & \mbf{\Phi}^{\sigma}_{xy} \end{bmatrix}$. Then, 
\begin{align}
&\max_{\substack{\left\|\mbf{w}^{\sigma}\right\|_{\infty} \leq \overline{w} \\ \left\|\mbf{v}^{\sigma}\right\|_{\infty} \leq \overline{v}}} \|\mbf{x}\|_{\infty} = \max_{\substack{\left\|\mbf{w}^{\sigma}\right\|_{\infty} \leq \overline{w} \\ \left\|\mbf{v}^{\sigma}\right\|_{\infty} \leq \overline{v}}} \left\|\mbf{\Phi^{\sigma}_{x}} \begin{bmatrix} \mbf{w}^{\sigma} \\ \mbf{v}^{\sigma} \end{bmatrix} \right\|_{\infty} \nonumber \\
= &\max_{\substack{\left\|\mbf{\tilde{w}}^{\sigma}\right\|_{\infty} \leq 1 \\ \left\|\mbf{\tilde{v}}^{\sigma}\right\|_{\infty} \leq 1}} \left\|\mbf{\Phi^{\sigma}_{x}} \begin{bmatrix} \overline{w}I_{n(T+1)} & \bm{0} \\ \bm{0} & \overline{v}I_{m(T+1)} \end{bmatrix} \begin{bmatrix} \mbf{\tilde{w}}^{\sigma} \\ \mbf{\tilde{v}}^{\sigma} \end{bmatrix} \right\|_{\infty} \nonumber \\
= & \left\|\mbf{\Phi_{x}^{\sigma}} \begin{bmatrix} \overline{w}I_{n(T+1)} & \bm{0} \\ \bm{0} & \overline{v}I_{m(T+1)} \end{bmatrix}\right\|_{\infty},
\end{align}
based on the definition of the induced norm. Then, \eqref{eq:L1} reduces to the following problem:
\begin{align} \label{eq:L1_phi}
\begin{split}
&\min_{\{\mbf{\Phi^{\sigma}}\}_{\sigma \in \mL}} \max_{\sigma \in \mL} \left\|\mbf{\Phi_{x}^{\sigma}} \begin{bmatrix} \overline{w}I_{n(T+1)} & \bm{0} \\ \bm{0} & \overline{v}I_{m(T+1)} \end{bmatrix}\right\|_{\infty} \\
&\text{subject to the constraints \eqref{eq:prefix_Phi}, \eqref{eq:Affine_sigma} for each $\sigma \in \mL$}.
\end{split}
\end{align}
This optimization can be written as a linear program in the decision variables $\mbf{\Phi}^{\sigma}$ since the set $\mL$ is finite.

We note that the optimization problems \eqref{eq:H2_phi} and \eqref{eq:L1_phi} result in a controller $\mbf{\Phi}^{\sigma}$ for each $\sigma \in \mL$. In practice, only the prefix $\sigma_{0:t}$ of the switching signal is known at each time instant. The controller at each time instant is thus determined by any of the controllers satisfying the prefix constraint \eqref{eq:prefix_K}. 

\begin{remark}
We remind the reader that the framework in this section can handle problems with any convex constraints of the form $\mbf{\Phi}^{\sigma} \in \mathcal{C}^{\sigma}$ for $\sigma \in \mL$. In addition, any objective function of the form $\mathbb{E}_{\sigma\in\mathcal{L}}[J^\sigma(\mathbf{\Phi}^\sigma)]$ or $\max_{\sigma\in\mathcal{L}}J^\sigma(\mathbf{\Phi}^\sigma)$ with convex $J^\sigma$ lead to convex programs since expectation and maximization operations preserve convexity.
\end{remark}


In the following section, we present two examples that illustrate how the optimal control problems \eqref{eq:H2} and \eqref{eq:L1} can be solved using the convex programs \eqref{eq:H2_phi} and \eqref{eq:L1_phi}.

\section{Examples} \label{sec:Examples}
In this section, we present our results on solving the optimal control Problems \ref{prob:H2} and \ref{prob:L1}. Our two examples illustrate the convex frameworks of Section \ref{sec:Reformulation} in applications to fault tolerance. Consider a fault occurring at any time instant $t_{\mathrm{fault}} \in \{0, \ldots, T\}$, which causes a system with nominal behavior to switch to malfunctioning behavior. We assume the system remains malfunctioning for the remaining time horizon. Thus, we have $M = 2$ systems in \eqref{eq:System_i}, for nominal and malfunctioning dynamics. {Accordingly, the language $\mL$ is chosen as in \eqref{eq:eg_L}, corresponding to the prefixes in Fig.~\ref{fig:FSM}(b).}
                

In these examples, we use the dynamics of {a subsystem of} the ADMIRE fighter jet model \cite{SDRA}, a common benchmark application for control frameworks \cite{Ola05, PBDO24}. 
The three states in this subsystem are the roll, pitch and yaw rates in $\mathrm{rad/s}$, denoted $p, q$ and $r$ respectively. The four inputs in this subsystem correspond to deflections (in radians) of canard wings, the left and right elevons and the rudder. A continuous-time, linearized model for this subsystem was established in \cite{Ola05}. We discretize these dynamics using a unit time step, obtaining the following time-invariant model:
\begin{gather}
\begin{split} \label{eq:ADMIRE}
x = \begin{bmatrix} p \\ q \\ r \end{bmatrix}; ~~ A = \begin{bmatrix} \phantom{-}0.3550 & 0 & 0.3428 \\ 0 & 0.6031 & 0 \\ -0.0521 & 0 & 0.7901 \end{bmatrix}; \\
B = \begin{bmatrix} 0 & -2.7200 & \phantom{-}2.7200 & \phantom{-}0.7376 \\ 1.298 & -0.9996 & -0.9996 & \phantom{-}0.0019 \\ 0 & -0.1153 & \phantom{-}0.1153 & -0.8362 \end{bmatrix}.
\end{split}
\end{gather}
We assume that the output $y_t$ corresponds to noisy measurements of the entire state vector, so that $C = I_3$. We note that due to linearization, the state and control inputs in this model correspond to deviations from their equilibrium values.





\subsection{Example 1: $\mathcal{H}_2$-optimal control, stochastic switching} \label{eg:H2}
In this example, we address the optimal control problem in \eqref{eq:H2} for the system \eqref{eq:ADMIRE}. {The nature of the fault we consider is an additional drift affecting the above dynamics. In the context of the model \eqref{eq:System_i}, this failure is expressed using the matrices $A^1 = A$, $A^2 = A - 1.5I_3$, $B^1 = B^2 = B$, $C^1 = C^2 =  I_3$, where $A$ and $B$ are given in \eqref{eq:ADMIRE}. Such a drift seeks to push the system towards instability.} In the formulation of Problem \ref{prob:H2}, we assume the process and measurement noise are distributed according to $\mbf{w}^{\sigma} \sim \mathcal{N}(\bm{0}, I)$ and $\mbf{v}^{\sigma} \sim \mathcal{N}(\bm{0}, I)$ for each $\sigma \in \mL$. We assume that a switch can occur with equal probability at any $t_{\mathrm{fault}} \in \{0, \ldots, T\}$, where the time horizon is $T = 10$. 
The cost matrices are $Q_t = I$ and $R_t = 2I$ for each $t$.

\begin{figure}[!t]
    \centering
    \includegraphics[width=0.4\textwidth]{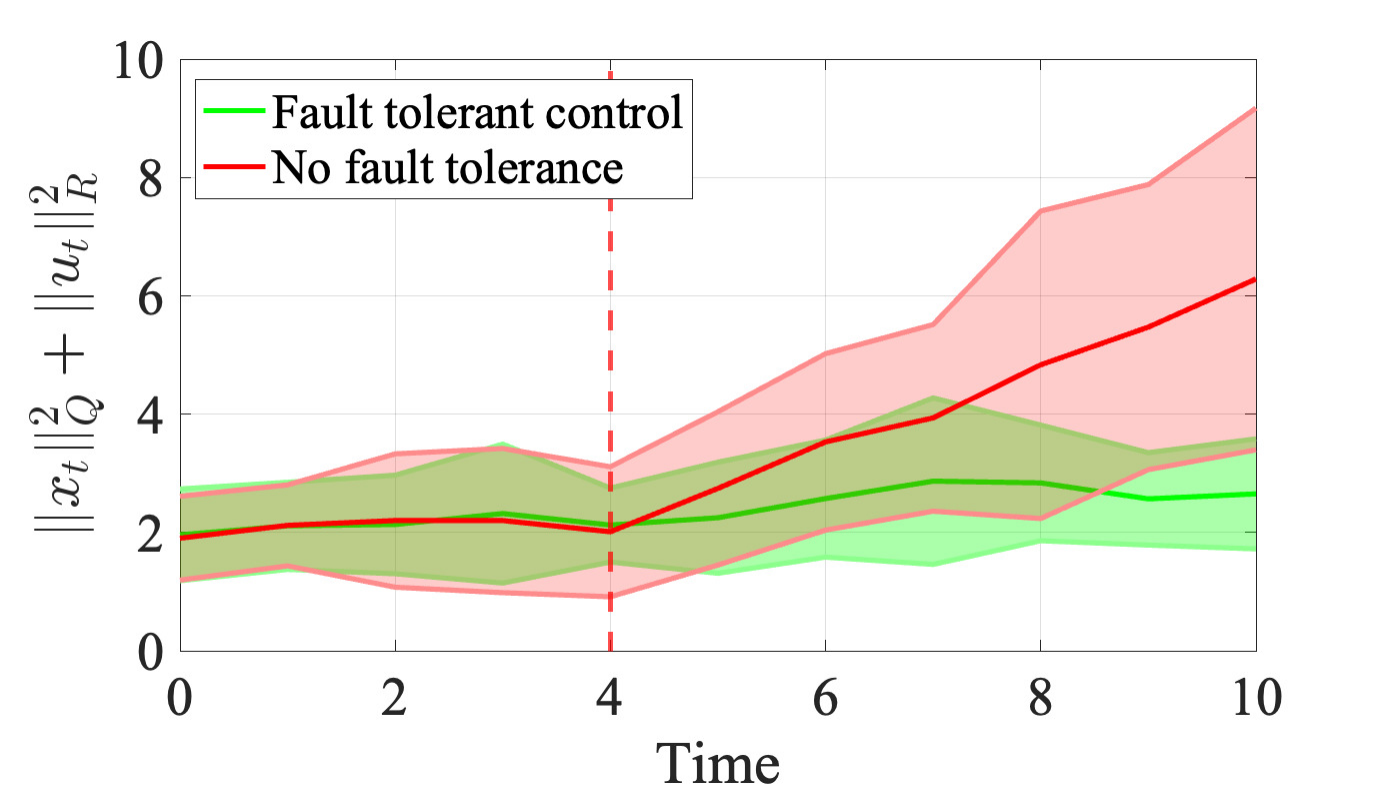}
    \caption{{Comparison of fault-tolerant and non-fault-tolerant controllers for Problem \ref{prob:H2}. The dashed red line denotes the time $t_{\mathrm{fault}}$ when the fault occurs. The shaded regions in both curves indicate trajectories within one standard deviation of the mean value of the trajectory.}}
    \label{fig:H2}
    \vspace{-0.6cm}
\end{figure}

{Fig.~\ref{fig:H2} illustrates trajectories of the cost $\|x_t\|_{Q}^{2} + \|u_t\|_{R}^{2}$ over time, where we show the mean and standard deviation of these trajectories over multiple runs. The vertical dashed red line at $t = 4$ denotes the switching time $t_{\mathrm{fault}}$, which is first randomly chosen and then fixed for all trajectories. We compare the performance of the fault-tolerant controller designed using the convex program in \eqref{eq:H2_phi}, with a baseline controller of the form \eqref{eq:u_SLS} which considers only nominal dynamics, and is hence not fault-tolerant. It is clear that the use of a mode-prefix-based controller improves performance over a standard optimal control approach. Once an additional drift affects the dynamics, the fault-tolerant approach is able to mitigate the effects of this fault while the baseline approach cannot do so. The fault-tolerant controller achieves both lower mean and variance of the cost over multiple runs.}




\subsection{Example 2: $\mathcal{L}_1$-optimal control, worst-case switching} \label{eg:L1}
In our second example, {we consider a sensor failure affecting the dynamics \eqref{eq:ADMIRE}, such that the first sensor produces measurements and the other sensors produce only measurement noise. This failure is expressed using the matrices $A^1 = A^2 = A$, $B^1 = B^2 = B$, $C^1 = I_3$ and $C^2 \in \mathbb{R}^{3\times 3}$, where $A$ and $B$ are given in \eqref{eq:ADMIRE} and $C^2$ is a matrix with $1$ in its $(1,1)$ element and zeros elsewhere.} We consider {non-}deterministic switches that are chosen to maximize the objective in \eqref{eq:L1} according to the formulation of Problem \ref{prob:L1}. Further, we assume bounded process and measurement noise, i.e., $\|\mbf{w}^{\sigma}\|_{\infty} \leq 1$ and $\|\mbf{v}^{\sigma}\|_{\infty} \leq 1$. We consider a time horizon $T = 10$. 

\begin{figure}[!t]
    \centering
    \includegraphics[width=0.4\textwidth]{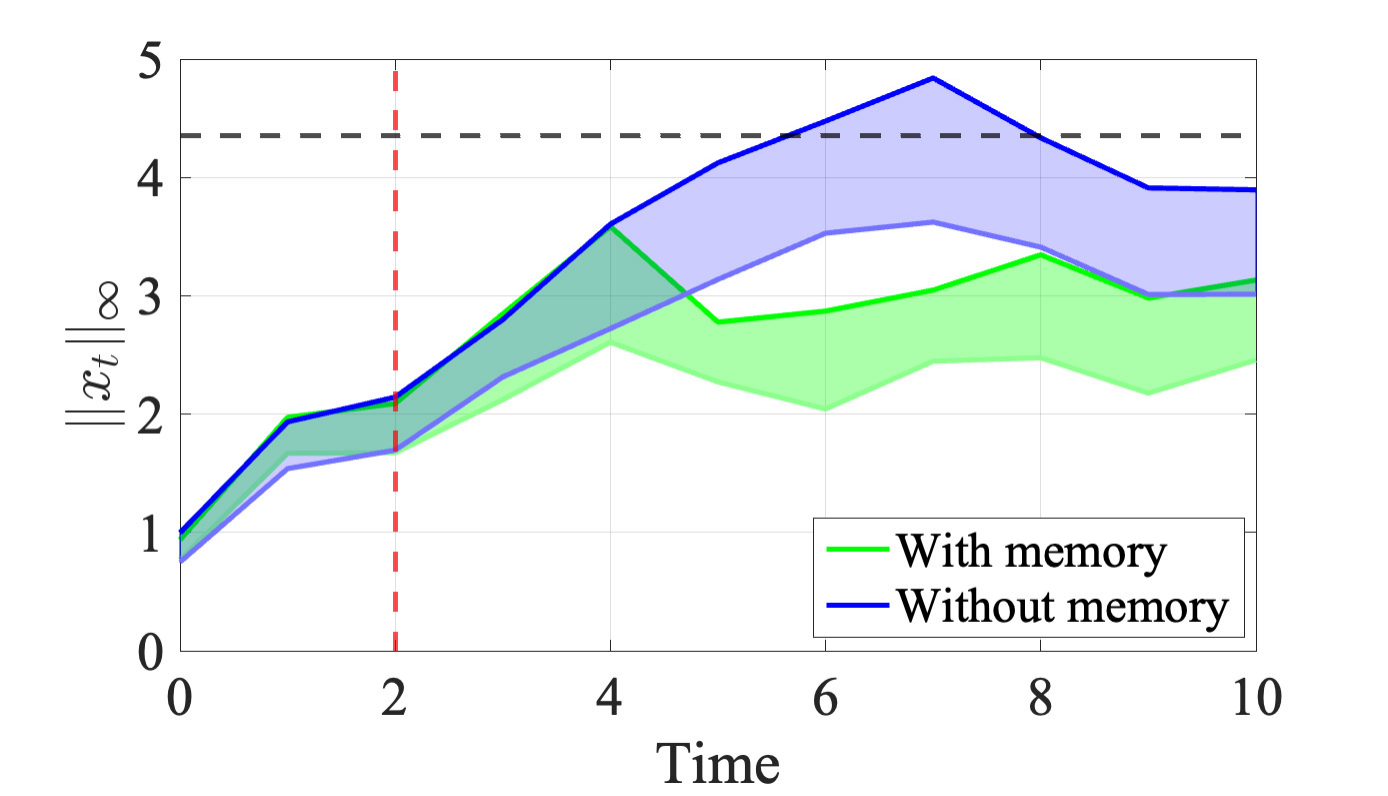}
    \caption{{Comparison of controllers with and without memory for Problem \ref{prob:L1}. The dashed red line denotes the time instant $t_{\mathrm{fault}} = 2$ when sensor failure occurs, corresponding to the worst-case $\sigma$ in \eqref{eq:L1_phi}. The dashed black lines in the state trajectory denote the maximal value of $\|\mbf{x}\|_{\infty}$ in \eqref{eq:L1}. The shaded regions indicate trajectories with values up to one standard deviation below the maximal trajectory.}}
    \label{fig:L1}
    \vspace{-0.6cm}
\end{figure}

{The results for this setting are as shown in Fig.~\ref{fig:L1} in terms of trajectories of $\|x_t\|_{\infty}$ over time. We compare the performance of the controller designed using the convex program in \eqref{eq:L1_phi}, with a static, memoryless controller that does not use past values of the output. In other words, $\mathbf{K}^{\sigma}$ is restricted to be block diagonal in this comparison. We show the maximum amplitude of this objective over multiple runs in the thick lines, and also indicate its values one standard deviation below this maximum. This variance among trajectories results from different values of the process and measurement noise, chosen from the interiors and vertices of the hypercubes $\|\mbf{w}^{\sigma}\|_{\infty} \leq 1$ and $\|\mbf{v}^{\sigma}\|_{\infty} \leq 1$. The vertical dashed red line denotes the worst-case switching time $t_{\mathrm{fault}} = 2$, corresponding to the worst-case switching signal $\sigma$ in \eqref{eq:L1_phi}.}

{From Fig.~\ref{fig:L1}, it is clear that the controller with memory outperforms a static, memoryless controller. The horizontal dashed black lines denote the worst-case amplitude of the state, based on the maximal value of $\|\mbf{x}\|_{\infty}$ in \eqref{eq:L1}, under the controller with memory. Then, we can see that all trajectories under the controller with memory are bounded within this worst-case value, while there is at least one trajectory under the memoryless controller that violates this constraint. This example thus demonstrates a safety certificate for the system state using the controller designed in the convex program \eqref{eq:L1_phi}, using the maximal value of $\|\mbf{x}\|_{\infty}$ in \eqref{eq:L1} computed \emph{a priori} on solving \eqref{eq:L1_phi}.}


We remark that while we consider two specific classes of failures in these examples, the framework developed in Sections \ref{sec:Formulation} and \ref{sec:Reformulation} can generally address a variety of switching mechanisms and system failures to ensure fault tolerance. However, the size of the language $\mL$ can become prohibitively large when more switching mechanisms are considered, increasing the computational cost of solving these problems.

\balance

\section{Conclusions} \label{sec:Conclusions}
In this paper, we presented an approach for designing optimal controllers for discrete-time switched linear systems over a finite horizon using system level synthesis. Motivated by fault tolerance for systems with abrupt changes in dynamics, we proposed a mode-prefix-based linear output-feedback controller that depends on the prefix, or history of the switching signal. To solve this problem, we synthesize a controller for each switching sequence under a prefix constraint to ensure that same prefixes lead to the same control input. We proved that such a prefix constraint is linear even in the parameterization through system level synthesis. We also demonstrated how the optimal control problems we consider simplify to convex programs. Using examples of a fighter jet model suffering from an additional drift or a sensor failure, we illustrate how this framework can be applied to ensure fault tolerance of control systems. A key avenue for future work is attempting to reduce the computational effort in solving these programs{, potentially through a forgetting mechanism limiting the available memory of measurements or prefix of the switching signal.}

\bibliographystyle{ieeetran}
\bibliography{references}

\end{document}